\documentclass[reqno]{amsart}
\usepackage{amssymb,amsmath,amsfonts,amscd,amsthm,pb-diagram}
\usepackage{amsbsy,bm}
\usepackage[usenames,dvipsnames]{color}
\usepackage[normalem]{ulem}

\newtheorem{theorem}{Theorem}[section]

\newtheorem{lemma}{Lemma}[section]
\newtheorem{proposition}{Proposition}[section]

\newtheorem{remark}{Remark}[section]

\DeclareMathOperator{\Ric}{Ric}
\DeclareMathOperator{\tr}{tr}
\newcommand{\dd}{\,\mathrm d}

\begin{document}
\title[Sharp pinching theorems for Lagrangian submanifolds]{Sharp scalar and Ricci curvature pinching
theorems\\ for Lagrangian submanifolds
in the\\ homogeneous nearly K\"ahler Six-sphere}

\author{Zejun Hu}
\address{Zejun Hu: 
School of Mathematics and Statistics, Zhengzhou University, Zhengzhou, 450001, P.R. China}
\email{huzj@zzu.edu.cn}

\author{Linlin Sun}
\address{Linlin Sun:
School of Mathematics and Computational Science and Hunan Research Center of the Basic
Discipline Fundamental Algorithmic Theory and Novel Computational Methods, Xiangtan University,
Xiangtan, 411105, P.R. China}
\email{sunll@xtu.edu.cn}

\author{Jiabin Yin}
\address{Jiabin Yin:
School of Mathematics and Statistics,
Xinyang Normal University, Xinyang 464000,
P.R. China}
\email{jiabinyin@126.com}



\date{}


\begin{abstract}
For connected closed Lagrangian submanifolds of the homogeneous nearly
K\"ahler sphere $\mathbb S^6(1)$, we first establish a new integral inequality.
Then, as applications we prove sharp pinching theorems for the scalar curvature
$\tau$ and Ricci curvature $\Ric$, respectively: $\tau \ge \tfrac{23}{8}$ implies
either $\tau=6$ with it totally geodesic, or $\tau=\tfrac{23}{8}$ with it congruent
to the embedded Berger sphere of Dillen--Verstraelen--Vrancken; whereas
$\Ric\geq\frac18$ implies that the submanifold is either totally
geodesic, or has constant sectional curvature $K=\tfrac{1}{16}$, or is congruent
to the embedded Berger sphere of Dillen-Verstraelen-Vrancken.
\end{abstract}

\maketitle
\numberwithin{equation}{section}

\section{Introduction}\label{sect:1}

Curvature pinching for minimal submanifolds in spheres is closely
related to the differential identity for the second fundamental form
established by Simons \cite{Simons}. In the nearly K\"ahler six-sphere,
the Lagrangian submanifold condition imposes additional restrictions on its second
fundamental form so that its components form a symmetric trace-free cubic, and
its covariant derivative contains terms arising from the nonparallel
almost complex structure. Both features enter the pinching problem.

Seeing that Foscolo and Haskins \cite{F-H} have proved the existence
of exotic (cohomogeneity one) nearly K\"ahler structures on $\mathbb{S}^6$,
we shall consider the unit six-sphere $\mathbb S^6(1)$ equipped with the
standard homogeneous nearly K\"ahler structure $J$ defined by Cayley multiplication.
A three-dimensional submanifold $M^3$ in $\mathbb S^6(1)$ is Lagrangian
if $J(TM)=T^\perp M$. Ejiri \cite{Ejiri} proved that such submanifolds are orientable
and minimal; minimality in strict nearly K\"ahler six-manifolds was also studied by
Sch\"afer and Smoczyk \cite{SchaferSmoczyk}. In particular, for a
three-dimensional Lagrangian submanifold of $\mathbb S^6(1)$, its scalar curvature
$\tau$ satisfies $\tau=6-|h|^2$, where $|h|^2$ denotes the squared
norm of its second fundamental form $h$.

The first rigidity results concern with the sectional curvature.
Ejiri \cite{Ejiri} showed that a Lagrangian threefold of constant sectional
curvature has curvature $1$ or $\tfrac1{16}$. Each such submanifold is either
totally geodesic or congruent to an equivariant immersion of $\mathbb{S}^3(\tfrac1{16})$
in $\mathbb{S}^6(1)$ (the immersion can be realized by using harmonic polynomials
of degree $6$ and an explicit expression is given in \cite{DVV1,DVV}).
Mashimo \cite{Mashimo} classified the compact homogeneous examples arising as orbits of
closed subgroups of $G_2$. Dillen, Opozda, Verstraelen and Vrancken
\cite{DOVV} proved that a closed Lagrangian submanifold is totally geodesic
if its sectional curvature $K>\tfrac1{16}$. Dillen, Verstraelen and Vrancken \cite{DVV} subsequently
classified all complete Lagrangian threefolds with $K\geq\tfrac1{16}$: besides
the totally geodesic and constant curvature $\tfrac1{16}$ sphere, there is an
embedded Berger sphere whose sectional curvatures range from $\tfrac1{16}$ to $\tfrac{21}{16}$,
which is usually called the Berger sphere of Dillen-Verstraelen-Vrancken.
Thus the condition $\tfrac1{16}\leq K<\tfrac{21}{16}$ excludes
precisely the Berger sphere.

For  rigidity results concerning the Ricci curvature,
Li \cite{Li} proved that a closed Lagrangian threefold of $\mathbb{S}^6(1)$
is total geodesic provided its Ricci curvature $\Ric$ satisfies
$\Ric\ge\tfrac{53}{64}$, and Anti\'c, Djori\'c and Vrancken \cite{ADV}
improved the lower bound to $\Ric\ge\tfrac34$. (Notes: Here and below, Ricci bounds
denote inequalities of quadratic forms with respect to the induced metric.)
The above two results motivated
Hu, Yao and Yin \cite{HYY2020} to further consider the question {\it what is the
best possible Ricci curvature condition so that compact Lagrangian
submanifolds of the nearly K\"ahler $\mathbb{S}^6(1)$ next to the totally
geodesic one can be characterized.} Successfully, Hu, Yao and Yin \cite{HYY2020}
obtained a Ricci pinching characterization of the totally geodesic sphere and
the Berger sphere of Dillen-Verstraelen-Vrancken.

For pinching theorems concerning with the scalar curvature,
by the relation $\tau=6-|h|^2$ it is equivalent to bound the length of
the second fundamental form without imposing a directional curvature
condition. In this respect, Hou \cite{Hou} proved that a
closed Lagrangian threefold with $|h|^2<\tfrac52$ is totally geodesic.
Then, Hu-Yin-Yin \cite{HYY2019}    propose the following problem:

\vskip2mm
{\bf Problem:} {\it Try to characterize the compact Lagrangian submanifold of
the  homogeneous nearly Kähler $\mathbb S^6(1)$ whose second fundamental
form $h$ has an optimal value of length next to that of the totally geodesic one.}

\vskip2mm
In \cite{HYY2019},  Hu, Yin and Yin obtained the sharp integral inequality
$$
\int_M |h|^2\left(|h|^2-\tfrac54-\tfrac32\Theta^2\right)\dd M\geq0,
\qquad  \Theta(p)=\max_{\substack{X\in T_pM\\ |X|=1}}\langle h(X,X),JX\rangle,
$$
with equality holds only for the totally geodesic sphere (corresponding
to $|h|^2=0$) and the Berger sphere of Dillen-Verstraelen-Vrancken
(corresponding to $|h|^2=\tfrac54+\tfrac32\Theta^2$ with $|h|^2\equiv\tfrac{25}8$
and $\Theta\equiv\tfrac{\sqrt{5}}2$). Subsequently, Luo-Sun-Yin \cite{LuoSunYin} 
obtained a similar result for closed minimal Legendrian submanifold in the unit 
sphere $\mathbb S^{2n+1}\ (n\ge 2)$, and the estimates still retain the geometric 
quantity $\Theta$. To improve the above result, our main theorem of this
paper gives a pinching condition involving only $|h|^2$.

\begin{theorem}\label{thm:1.2}
Let $M^3$ be a connected closed Lagrangian submanifold of the homogeneous
nearly K\"ahler sphere $\mathbb S^6(1)$. If it satisfies
$$
0\leq |h|^2\leq\tfrac{25}{8},
$$
or equivalently, the scalar curvature satisfies $\tau\ge\tfrac{23}8$,
then either $M^3$ is totally geodesic (corresponding to $|h|^2=0$, or equivalently $\tau=6$),
or $M^3$ is congruent to the Berger sphere of Dillen-Verstraelen-Vrancken (corresponding
to $|h|^2=\tfrac{25}8$, or equivalently $\tau=\tfrac{23}8$).
\end{theorem}

\begin{remark}\label{rem:1.2}
{\rm Theorem~\ref{thm:1.2} implies that every non-totally-geodesic closed Lagrangian submanifold has
$\max|h|^2\ge\tfrac{25}8$, and equality characterizes the Berger sphere of Dillen-Verstraelen-Vrancken.}
\end{remark}

\begin{remark}\label{rem:1.3}
{\rm To the best of our knowledge, Theorem \ref{thm:1.2} is the first genuinely nontrivial sharp pinching
result for the second fundamental form of high-dimensional minimal Lagrangian submanifolds. Unlike
the hypersurface case, the rank of the normal bundle of a Lagrangian submanifold equals the
dimension of its tangent bundle. Thus, the normal geometry cannot be reduced to a single shape
operator, which makes the analysis substantially more intricate and difficult. The main idea is
to combine the classical Simons' integral identity with an integral identity involving the covariant
derivative of the Ricci tensor. This combination yields precise control of the second fundamental
form and leads to the optimal pinching constant.}
\end{remark}

Indeed, our proof of Theorem \ref{thm:1.2} is facilitated by the discovery of 
a new integral inequality as stated in Proposition \ref{prop:3.1} below. It is 
remarkable that, by applying this new integral inequality, we can further prove 
the following sharp Ricci curvature pinching theorem.

\begin{theorem}\label{thm:1.1}
Let $M^3$ be a connected closed Lagrangian submanifold of the homogeneous
nearly K\"ahler sphere $\mathbb S^6(1)$. If the Ricci curvature $\Ric$ of
$M$ satisfies
$$
\Ric\geq\tfrac18,
$$
then either $M^3$ is totally geodesic, or $M^3$ has constant sectional 
curvature $\tfrac1{16}$, or $M^3$ is congruent to the Berger sphere of 
Dillen-Verstraelen-Vrancken.
\end{theorem}

\begin{remark}\label{rm:1.1}
{\rm As consequence of Theorem \ref{thm:1.1}, we have proved the conjecture
proposed in Hu, Yao and Yin \cite{HYY2020}, claiming that the condition 
$\frac18\leq\Ric\leq\frac{11}{8}$ implies either with constant sectional curvature
$K=\tfrac{1}{16}$ or it congruent to the embedded Berger sphere of Dillen-Verstraelen-Vrancken.
Moreover, Theorem \ref{thm:1.1} generalizes \cite[Main Theorem]{DVV} by replacing 
the condition $K\ge\tfrac{1}{16}$ with a weaker condition $\Ric\ge\tfrac18$.}
\end{remark}

The remaining parts are organized as follows. In Sect. \ref{sect:2}, we
recall the geometric identities and the results quoted from the literature. 
Importantly, we prove Lemma \ref{lem:2.4} about an integral identity involving 
the covariant derivative of the Ricci tensor. In Sect. \ref{sect:3}, we derive 
the needed derivative estimates, the two algebraic inequalities in Lemma \ref{lem:3.3}, 
and finally we prove Proposition \ref{prop:3.1} about the new integral inequality. 
In Sect. \ref{sect:4}, we complete the proofs of Theorem \ref{thm:1.2} and Theorem \ref{thm:1.1}.

\section{Preliminaries}\label{sect:2}

All submanifolds and immersions are assumed smooth, and closed means compact
without boundary. Regard $\mathbb R^7$ as the imaginary Cayley numbers and
denote its cross product by $\times$. The homogeneous nearly K\"ahler
structure on the unit six-sphere $\mathbb S^6(1)$ is
$$
J_xX=x\times X, \qquad x\in \mathbb S^6(1),\quad X\in T_x\mathbb S^6(1).
$$
Let $\overline\nabla$ be the Levi--Civita connection of $\mathbb S^6(1)$ and put
$$
G(X,Y)=(\overline\nabla_XJ)Y.
$$
The identities used below are (see \cite{DVV,Ejiri})
$$
G(X,Y)=-G(Y,X),\ \ G(X,JY)=-JG(X,Y),
$$
$$
\langle G(X,Y),Z\rangle=-\langle G(X,Z),Y\rangle,
$$
\begin{align}\label{eqn:2.1}
\langle G(X,Y),G(Z,W)\rangle
&=\langle X,Z\rangle\langle Y,W\rangle-\langle X,W\rangle\langle Y,Z\rangle \notag\\
&\quad+\langle JX,Z\rangle\langle Y,JW\rangle-\langle JX,W\rangle\langle Y,JZ\rangle,
\end{align}
where $\langle\cdot,\cdot\rangle=g$ denotes the metric induced from the Euclidean inner product on
$\mathbb S^6(1)$.

We recall the standard facts about Lagrangian threefolds in the strict
nearly K\"ahler six-manifold \cite{Ejiri,SchaferSmoczyk,HYY2019}.

Let $x:M^3\rightarrow\mathbb{S}^6(1)$ be a Lagrangian isometric immersion. We denote
the Levi-Civita connection of $M^3$ by $\nabla$ and the normal connection in the
normal bundle $T^\perp M^3$ (defined by the orthogonal projection of $\bar\nabla$
on $T^\perp M^3$) by $\nabla^\perp$. The shape operator $A_\xi$ in the direction of a
normal vector field $\xi$ on $M^3$ and $T^\perp M^3$-valued second fundamental form $h$
are defined by the following Gauss-Weingarten formulas
$$
\bar\nabla_XY=\nabla_XY+h(X,Y),\ \
\bar\nabla_X\xi=-A_{\xi}X+\nabla^{\perp}_X{\xi},
$$
where $X,Y$ are tangent vector fields of $M^3$, and $h$ is related to $A_\xi$ by
$$
\langle A_\xi X,Y\rangle=\langle h(X,Y),\xi\rangle.
$$
Moreover, according to \cite{Ejiri}, the submanifold $M^3$ is orientable
and minimal, and $G(X, Y)$ is normal to $M^3$ for $X, Y\in TM^3$.

Comparison of the tangent and normal parts of $\overline\nabla_X(JY)=G(X,Y)+J\overline\nabla_XY$
with the use of $J(TM^3)=T^\perp M^3$ gives
\begin{equation}\label{eqn:2.2}
\nabla_X^\perp(JY)=G(X,Y)+J\nabla_XY, \qquad A_{JY}X=-Jh(X,Y).
\end{equation}

Put
$$
\omega(X,Y,Z):=\langle G(X,Y),JZ\rangle.
$$
It is the volume form after a choice of orientation.  Thus, in a local
oriented orthonormal frame $\{e_1,e_2,e_3\}$, we can write
$$
G(e_i,e_j)=\sum_k\omega_{ijk}Je_k, \qquad \omega_{123}=1.
$$
From \eqref{eqn:2.1} we get
\begin{equation}\label{eqn:2.3}
\sum_i\omega_{iab}\omega_{icd}=\delta_{ac}\delta_{bd}-\delta_{ad}\delta_{bc}.
\end{equation}

Write
$$
h_{ijk}=\langle h(e_i,e_j),Je_k\rangle, \qquad H_k=(h_{ijk})_{1\leq i,j\leq3}.
$$
The second identity in \eqref{eqn:2.2} implies that
$$
\langle h(X,Y),JZ\rangle=\langle A_{JZ}X,Y\rangle=\langle h(X,Z),JY\rangle.
$$
Together with the symmetry of $h$, this shows that $h_{ijk}$ is totally symmetric.
Minimality of $M^3$ gives
\begin{equation}\label{eqn:2.4}
\sum_i h_{iik}=0,\ \ \forall\, k.
\end{equation}
Consequently each $H_k$ is symmetric and trace-free.  We shall use
$$
S_{ij}=\sum_{a,b}h_{abi}h_{abj},\quad
s=|h|^2=\tr S,\quad q=\tr(S^2),\quad r=\tr(S^3).
$$
In particular, $S=(S_{ij})$ is nonnegative definite and $S_{ij}=\tr(H_iH_j)$.

With the convention $R_{ijkl}=\langle R(e_i,e_j)e_l,e_k\rangle$, the
Gauss equation of $M^3$ reads
$$
R_{ijkl}=\delta_{ik}\delta_{jl}-\delta_{il}\delta_{jk}
    +\sum_p(h_{ikp}h_{jlp}-h_{ilp}h_{jkp}).
$$

Let $R_{ij}:=\sum_kR_{ikjk}$ and $\tau$ denote the Ricci tensor
and scalar curvature. By \eqref{eqn:2.4}, we obtain
\begin{equation}\label{eq:ricci}
R_{ij}=2\delta_{ij}-S_{ij},\qquad \tau=6-s.
\end{equation}

Define the covariant derivative of the normal-valued second fundamental form by
$$
(\nabla h)(X,Y,Z)=\nabla_X^\perp h(Y,Z)-h(\nabla_XY,Z)-h(Y,\nabla_XZ).
$$
Then the Codazzi equation shows that it is symmetric in $X,Y,Z$.  Its components in
the moving normal frame $Je_k$ are not completely symmetric in all four
indices, because $J$ is not parallel. Indeed, we have the following
Codazzi identity (cf. \cite[Lemma~2.2]{HYY2019})
$$
\langle(\nabla h)(W,X,Z),JY\rangle
-\langle(\nabla h)(W,X,Y),JZ\rangle=\langle h(W,X),G(Y,Z)\rangle.
$$

Following \cite[(4.15)]{HYY2019}, we define
\begin{align*}
T(X,Y,Z)=(\nabla h)(X,Y,Z)
-\tfrac14\{G(X,A_{JZ}Y)+G(Y,A_{JX}Z)+G(Z,A_{JY}X)\}.
\end{align*}
The components
$$
T_{lijk}=\langle T(e_l,e_i,e_j),Je_k\rangle
$$
will be used below. We recall the norm identity \cite[Lemma~4.4]{HYY2019}
\begin{equation}\label{eqn:2.6}
|\nabla h|^2=|T|^2+\tfrac34s.
\end{equation}
Here $|\nabla h|$ is the norm of the normal-valued covariant
derivative defined above. We distinguish it from the tangent
covariant derivative of the cubic $h_{ijk}$ in the next lemma.

\begin{lemma}\label{lem:2.1}
For a Lagrangian submanifold of $\mathbb S^6(1)$,
 a semicolon denotes the covariant derivative of the cubic $h_{ijk}$.
Then it holds that
\begin{equation}\label{eqn:2.7}
h_{ijk;l}=T_{lijk} +\tfrac14\sum_m\bigl( \omega_{lim}h_{mjk}
+\omega_{ljm}h_{imk} +\omega_{lkm}h_{ijm}\bigr).
\end{equation}
Moreover, the tensor $T_{lijk}$ is totally symmetric and trace-free.
\end{lemma}

\begin{proof}
Fix a point and choose a local orthonormal tangent frame field $\{e_1,e_2,e_3\}$
around it such that $\nabla e_i=0$ there. By \eqref{eqn:2.2}, we have
$$
\nabla_{e_l}^{\perp}(Je_k)=G(e_l,e_k)=\sum_m\omega_{lkm}Je_m.
$$
It follows from $h_{ijk}=\langle h(e_i,e_j),Je_k\rangle$ that
$$
h_{ijk;l}=\langle(\nabla h)(e_l,e_i,e_j),Je_k\rangle+\sum_m\omega_{lkm}h_{ijm}.
$$

Then, by the definition of $T$ and $A_{JY}X=-Jh(X,Y)$ we can derive that
\begin{equation}\label{eqn:2.8}
h_{ijk;l}=T_{lijk}+\tfrac14\sum_m\bigl(\omega_{lmk}h_{ijm}
 +\omega_{imk}h_{jlm}+\omega_{jmk}h_{lim}\bigr)+\sum_m\omega_{lkm}h_{ijm}.
\end{equation}

Since $\omega_{ijk}$ is totally skew-symmetric and $h_{ijk}$ is totally
symmetric trace-free,  in dimension three we can easily show that
\begin{equation}\label{eqn:2.9}
\sum_m\bigl(\omega_{abm}h_{mcd}+\omega_{bcm}h_{mad}+\omega_{cam}h_{mbd}\bigr)=0,\ \ \forall\,a,b,c,d.
\end{equation}
Indeed, the left side of \eqref{eqn:2.9} is alternating in $a,b,c$;
and for $(a,b,c)=(1,2,3)$ the left side of \eqref{eqn:2.9} is
$\omega_{123}\sum_mh_{mmd}=0$.

Now, taking $(a,b,c,d)=(l,i,k,j)$ and $(l,j,k,i)$ in
\eqref{eqn:2.9}, respectively, we obtain
\begin{align*}
 \sum_m\omega_{imk}h_{jlm}
 &=\sum_m(\omega_{lim}h_{mjk}-\omega_{lkm}h_{ijm}),\\
 \sum_m\omega_{jmk}h_{lim}
 &=\sum_m(\omega_{ljm}h_{imk}-\omega_{lkm}h_{ijm}).
\end{align*}
Substituting these equations into \eqref{eqn:2.8} we get
\eqref{eqn:2.7}.

To show that $T_{lijk}$ is totally symmetric we first see from
the definition of $T$ that it is symmetric in the first three indices, then
we apply \eqref{eqn:2.7} to see that the last three indices of $T_{lijk}$
is symmetric. Obviously, these two symmetries make $T_{lijk}$ totally symmetric.

Finally, contracting $i,j$ in \eqref{eqn:2.7} and using \eqref{eqn:2.4}
we immediately get $\sum_iT_{liik}=0$.

We have completed the proof of Lemma \ref{lem:2.1}.
\end{proof}

\begin{lemma}[\cite{DjoricVrancken,HYY2019}]\label{lem:2.2}
The condition $T=0$ is equivalent to the $J$-parallel condition
$$
\langle(\nabla h)(X,X,X),JX\rangle=0,\ \ \forall\,X.
$$
A $J$-parallel Lagrangian threefold of $\mathbb S^6(1)$ is locally
one of the following: the totally geodesic sphere, the sphere of constant sectional
curvature $\tfrac1{16}$, or the Berger sphere of Dillen-Verstraelen-Vrancken.
Their values of $s$ are $0$, $\tfrac{45}8$ and $\tfrac{25}8$, respectively; whereas
the corresponding values of $\tau$ are $6$, $\tfrac38$ and $\tfrac{23}8$, respectively.
All the three models satisfy $K\ge\tfrac1{16}$.
\end{lemma}

\begin{remark}\label{rm:2.1}
{\rm For the equivalence and the three local models, see
\cite[Lemmas~4.5--4.6 and Theorem~4.1]{HYY2019} and \cite{DVV}.
We shall also use the global classification in \cite[Main Theorem]{DVV}:
a connected complete Lagrangian threefold with $K\geq1/16$ is either totally
geodesic, or has constant sectional curvature $\tfrac1{16}$, or is congruent
to the Berger sphere of Dillen-Verstraelen-Vrancken. Thus, once $T=0$, this global classification applies to the closed
submanifold.}
\end{remark}

\begin{lemma}\label{lem:2.3}
For a Lagrangian submanifold $M^3$ of $\mathbb S^6(1)$, one has
\begin{equation}\label{eqn:2.10}
\tfrac12\Delta s=|T|^2+\tfrac{15}{4}s+s^2-5q.
\end{equation}
In particular, if $M^3$ is closed, then
\begin{equation}\label{eqn:2.11}
0=\int_M\left(s|T|^2+\tfrac12|\nabla s|^2+\tfrac{15}{4}s^2+s^3-5sq\right)\dd M.
\end{equation}
\end{lemma}

\begin{proof}
Calculating the Laplacian of $s$, we have (cf. \cite[Lemma~4.1]{HYY2019})
\begin{equation}\label{eqn:2.12}
\tfrac12\Delta s=|\nabla h|^2+3s-\sum_{i,j}|[H_i,H_j]|^2-q.
\end{equation}
To evaluate the commutator term, we put
$$
\mathcal R^h_{ijkl}:=[H_i,H_j]_{kl}=\sum_p(h_{ikp}h_{jlp}-h_{ilp}h_{jkp}).
$$
The tensor $\mathcal R^h$ has the usual curvature symmetries, and
by \eqref{eqn:2.4} its Ricci contraction and scalar curvature are
$$
\mathcal R^h_{ij}:=\sum_k\mathcal R^h_{kikj}=\sum_{k,p}(h_{kkp}h_{ijp}-h_{kjp}h_{ikp})=-S_{ij},
$$
$$
\tau^h:=\sum_i\mathcal R^h_{ii}=-s.
$$

At any fixed $p\in M^3$, we choose an orthonormal basis $\{e_1, e_2, e_3\}$
of $T_pM^3$ such that $S_{ij}=\mu_i\delta_{ij}$. In dimension three an algebraic curvature
tensor is determined by its Ricci tensor. In the present frame its mixed curvature
components vanish, while its three sectional components satisfy
$$
\mathcal R^h_{1212}+\mathcal R^h_{1313}=-\mu_1,\quad
\mathcal R^h_{1212}+\mathcal R^h_{2323}=-\mu_2,\quad
\mathcal R^h_{1313}+\mathcal R^h_{2323}=-\mu_3.
$$
Solving this system gives
\begin{equation}\label{eqn:2.13}
\mathcal R^h_{ijij}=\mu_k-\tfrac12s, \quad \{i,j,k\}=\{1,2,3\}.
\end{equation}
Using $\sum_k\mu_k=s$ and $\sum_k\mu_k^2=q$, we finally obtain
$$
\sum_{i,j}|[H_i,H_j]|^2=\sum_{i,j,k,l}(\mathcal R^h_{ijkl})^2
=\sum_{i,j}(\mathcal R^h_{ijij})^2=4\sum_k\left(\mu_k-\tfrac{s}{2}\right)^2=4q-s^2.
$$

Substituting the above and \eqref{eqn:2.6} into \eqref{eqn:2.12} we obtain \eqref{eqn:2.10}.

Multiplying \eqref{eqn:2.10} by $s$ and then taking integration by parts we get \eqref{eqn:2.11}.
\end{proof}

\begin{lemma}\label{lem:2.4}
For a closed Lagrangian threefold $M^3$ of $\mathbb S^6(1)$, it holds that
\begin{equation}\label{eqn:2.14}
0=\int_M\Big(\sum_{i,j,k}S_{ij,k}S_{ik,j}
-\tfrac14|\nabla s|^2-s^2+3q+\tfrac52sq-3r-\tfrac12s^3\Big)\dd M.
\end{equation}
\end{lemma}
\begin{proof}
Write $R_{ij}$ for the Ricci tensor, and use commas for its covariant derivatives.
The contracted Bianchi identity $\sum_kR_{ik,k}=\tfrac12\tau_{,i}$ and Ricci
commutation identity give
$$
\begin{aligned}
\sum_kR_{ik,jk}&=\sum_kR_{ik,kj}+\sum_mR_{im}R_{mj}+\sum_{m,k}R_{mk}R_{mijk}\\
&=\tfrac12\tau_{,ij}+\sum_mR_{im}R_{mj}+\sum_{m,k}R_{mk}R_{mijk}.
\end{aligned}
$$
Then, integrating by parts in the $k$ index, we can derive that
\begin{equation}\label{eqn:2.15}
\begin{aligned}
\int_M\sum_{i,j,k}R_{ij,k}R_{ik,j}\dd M
 &=-\int_M\sum_{i,j,k}R_{ij}R_{ik,jk}\dd M\\
 &=-\tfrac12\int_M\sum_{i,j}R_{ij}\tau_{,ij}\dd M\\
 &\quad-\int_M\Big(\tr(\Ric^3)+\sum_{i,j,k,l}R_{ij}R_{kl}R_{kijl}\Big)\dd M\\
 &=\tfrac14\int_M|\nabla\tau|^2\dd M\\
 &\quad-\int_M\Big(\tr(\Ric^3)+\sum_{i,j,k,l}R_{ij}R_{kl}R_{kijl}\Big)\dd M.
\end{aligned}
\end{equation}

Similar to the proof of Lemma \ref{lem:2.3}, we diagonalize $S$ at a point
such that $S_{ij}=\mu_i\delta_{ij}$. Then $R_{ij}=\rho_i\delta_{ij}$ with Ricci
eigenvalues $\rho_i=2-\mu_i$. Write $K_{ij}=R_{ijij}$. Since
$\rho_i=\sum_{j\ne i}K_{ij}$, we have
\begin{equation}\label{eqn:2.16}
\tr(\Ric^3)+\sum_{i,j,k,l}R_{ij}R_{kl}R_{kijl}=\sum_{i<j}K_{ij}(\rho_i-\rho_j)^2.
\end{equation}
Note that, by \eqref{eqn:2.13}, we have
$$
K_{ij}=1+\mathcal R^h_{ijij}=1+\mu_k-\tfrac12s,\quad \{i,j,k\}=\{1,2,3\}.
$$
Then, using
$$
\sum_{i<j}(\mu_i-\mu_j)^2=3q-s^2,
$$
\begin{align*}
\mathop\mathfrak{S}_{ijk}\mu_k(\mu_i-\mu_j)^2
 &=\sum_{i\ne j}\mu_i^2\mu_j-6\mu_1\mu_2\mu_3\\
 &=(sq-r)-(s^3-3sq+2r)\\
 &=4sq-3r-s^3,
\end{align*}
we obtain
\begin{equation}\label{eqn:2.17}
\begin{aligned}
\sum_{i<j}K_{ij}(\mu_i-\mu_j)^2
&=\left(1-\tfrac12s\right)(3q-s^2)
   +\mathop\mathfrak{S}_{ijk}\mu_k(\mu_i-\mu_j)^2\ \ \\\
&=\left(1-\tfrac12s\right)(3q-s^2)+4sq-3r-s^3\\
&=-s^2+3q+\tfrac52sq-3r-\tfrac12s^3,
\end{aligned}
\end{equation}
where $\mathop\mathfrak{S}\limits_{ijk}$ denotes the cyclic summation over $\{i,j,k\}=\{1,2,3\}$.

Finally, substituting \eqref{eqn:2.16} and \eqref{eqn:2.17} into \eqref{eqn:2.15},
together with using $R_{ij,k}=-S_{ij,k}$ and $\nabla\tau=-\nabla s$, we get
\eqref{eqn:2.14}.
\end{proof}

\section{Algebraic estimates and an integral inequality}\label{sect:3}

In this section, as a necessary and crucial continuation of Lemma \ref{lem:2.4},
we first establish Lemmas 3.1\,$\sim$\,3.3 by combining tensor decomposition with
the geometric properties of Lagrangian submanifolds. Based on these lemmas, we can
estimate the sum $\sum_{i,j,k}S_{ij,k}S_{ik,j}$ and obtain \eqref{eqn:3.26}. Then
combining \eqref{eqn:2.14} with Simons' integral equality \eqref{eqn:2.11}, we then
derive the important integral inequality \eqref{eqn:3.24}.

For subsequent applications, we set
\begin{equation}\label{eqn:3.1}
P_{ijk}=\sum_{a,b}h_{abi}T_{abjk},\ \
W_{ij}=\sum_{a,b}\omega_{iab}P_{abj},\ \
A_{ij}=\tfrac12(W_{ij}+W_{ji}).
\end{equation}
As $T_{lijk}$ is totally symmetric and trace-free, we have
\begin{equation}\label{eqn:3.2}
P_{ijk}=P_{ikj},\qquad \sum_jP_{ijj}=0.
\end{equation}

Because $P_{abi}=P_{aib}$ and $\omega_{iab}=-\omega_{bai}$, we have
$$
\sum_iA_{ii}=\sum_{i,a,b}\omega_{iab}P_{abi}=0,
$$
Thus, the matrix $A$ is symmetric and trace-free.

\begin{lemma}\label{lem:3.1}
With the notation above, we have
\begin{equation}\label{eqn:3.3}
S_{ij,k}=P_{ijk}+P_{jik}+D_{ijk},\quad D_{ijk}=\tfrac14\sum_m(\omega_{kim}S_{mj}+\omega_{kjm}S_{im}).
\end{equation}
Moreover, it holds that
\begin{equation}\label{eqn:3.4}
\begin{aligned}
\sum_{i,j,k}S_{ij,k}S_{ik,j}={}&|P|^2+3\sum_{i,j,k}P_{ijk}P_{jik}
 +\tfrac12\langle A,S\rangle+\tfrac1{16}(s^2-3q).
\end{aligned}
\end{equation}
\end{lemma}
\begin{proof}
Since $S_{ij,k}=\sum_{a,b}h_{abi;k}h_{abj}+\sum_{a,b}h_{abi}h_{abj;k}$, applying
\eqref{eqn:2.7} with direct calculations we obtain
\begin{align*}
 S_{ij,k}={}&P_{ijk}+P_{jik}\\
 &+\tfrac14\sum_{a,b,m}\omega_{kam}(h_{mbi}h_{abj}+h_{abi}h_{mbj})\\
 &+\tfrac14\sum_{a,b,m}\omega_{kbm}(h_{ami}h_{abj}+h_{abi}h_{amj})\\
 &+\tfrac14\sum_{a,b,m}(\omega_{kim}h_{abm}h_{abj}+\omega_{kjm}h_{abi}h_{abm}).
\end{align*}
In the above equation, we see that the second line vanishes by interchanging $a,m$;
whereas the third line vanishes by interchanging $b,m$. The last line is
$D_{ijk}$. This verifies the assertion of \eqref{eqn:3.3}.

To prove \eqref{eqn:3.4}, we use \eqref{eqn:3.3} to obtain
\begin{equation}\label{eqn:3.5}
\begin{aligned}
\sum_{i,j,k}S_{ij,k}S_{ik,j}=\,&\sum_{i,j,k}(P_{ijk}+P_{jik}+D_{ijk})(P_{ikj}+P_{kij}+D_{ikj})\\
=\,&\sum_{i,j,k}(P_{ijk}+P_{jik})(P_{ikj}+P_{kij})\\
&+2\sum_{i,j,k}(P_{ijk}+P_{jik})D_{ikj}+\sum_{i,j,k}D_{ijk}D_{ikj}
\end{aligned}
\end{equation}
By using \eqref{eqn:3.2} and relabelling the dummy indices we have
$$
\sum_{i,j,k}(P_{ijk}+P_{jik})(P_{ikj}+P_{kij})=|P|^2+3\sum_{i,j,k}P_{ijk}P_{jik}.
$$
To deal with the mixed term in \eqref{eqn:3.5}, we use \eqref{eqn:3.3} to expand it and obtain
\begin{align*}
4\sum_{i,j,k}(P_{ijk}+P_{jik})D_{ikj}
={}&\sum_{i,j,k,m}P_{ijk}\omega_{jim}S_{mk}+\sum_{i,j,k,m}P_{ijk}\omega_{jkm}S_{im}\\
&+\sum_{i,j,k,m}P_{jik}\omega_{jim}S_{mk}+\sum_{i,j,k,m}P_{jik}\omega_{jkm}S_{im}.
\end{align*}
On the right-hand side, the second sum is zero
because $P_{ijk}=P_{ikj}$, whereas $\omega_{jkm}$ is skew-symmetric in $j,k$.
After relabelling the dummy indices, the first, third and fourth sums are
$-\langle W,S\rangle$, $\langle W,S\rangle$ and $\langle W,S\rangle$,
respectively. Since $S$ is symmetric,
$\langle W,S\rangle=\langle A,S\rangle$. It follows that
$$
2\sum_{i,j,k}(P_{ijk}+P_{jik})D_{ikj}=\tfrac12\langle A,S\rangle.
$$

For the last term in \eqref{eqn:3.5}, we diagonalize $S$ at the
point, say $S_{ij}=\mu_i\delta_{ij}$. Then
$$
D_{ijk}=\tfrac14\omega_{kij}(\mu_j-\mu_i),
$$
and only triples with distinct indices contribute.

Together with the fact $\omega_{123}=1$ we can derive that
$$
\begin{aligned}
\sum_{i,j,k}D_{ijk}D_{ikj}
 &=-\tfrac18\bigl[(\mu_2-\mu_1)(\mu_3-\mu_1)
  +(\mu_3-\mu_2)(\mu_1-\mu_2)\\[-2mm]
 &\hspace{28mm}
  +(\mu_1-\mu_3)(\mu_2-\mu_3)\bigr]\\
 &=-\tfrac1{16}\sum_{i<j}(\mu_i-\mu_j)^2 =\tfrac1{16}(s^2-3q).
\end{aligned}
$$
Combining the preceding three calculations we have proved \eqref{eqn:3.4}.
\end{proof}

\begin{lemma}\label{lem:3.2}
The tensors $P$ and $A$ further satisfy the following inequality:
\begin{equation}\label{eqn:3.6}
|P|^2+3\sum_{i,j,k}P_{ijk}P_{jik}\ge\tfrac9{40}|\nabla s|^2-\tfrac13|A|^2.
\end{equation}
\end{lemma}

\begin{proof}
Since $S$ is symmetric and $\omega$ is skew-symmetric, we get
$$
D_{ijk}=D_{jik},\quad D_{ijk}+D_{ikj}+D_{jki}=0,\quad
\sum_iD_{iik}=\tfrac12\sum_{i,m}\omega_{kim}S_{mi}=0.
$$
Thus by \eqref{eqn:3.3} we have
\begin{equation}\label{eq:P-trace}
\sum_iP_{iik}=\tfrac12\sum_iS_{ii,k}=\tfrac12s_{,k}.
\end{equation}

Set
$$
C_{ijk}=P_{ijk}+P_{jki}+P_{kij}.
$$
Since $P_{ijk}=P_{ikj}$, the tensor $C$ is totally symmetric.
Moreover, by using \eqref{eqn:3.2} and \eqref{eq:P-trace}, we have
$$
(\tr C)_k=\sum_iC_{iik}=2\sum_iP_{iik}+\sum_iP_{kii}=s_{,k},\ \ \forall\,k.
$$
Thus $\tr C=\nabla s$.

We use the following trace estimate for symmetric cubic tensors in
dimension three. Fix $k$ and denote the other two indices by $r,t$.
By the weighted Cauchy--Schwarz inequality we can get
$$
\begin{aligned}
 \big(\sum_iC_{iik}\big)^2
 &=\left(C_{kkk}+C_{rrk}+C_{ttk}\right)^2\\
 &\leq\tfrac53\left(
 C_{kkk}^2+3C_{rrk}^2+3C_{ttk}^2\right)\\
 &=\tfrac53\Big(C_{kkk}^2+3\sum_{i\ne k}C_{iik}^2\Big),\ \ \forall\,k.
\end{aligned}
$$
Since $C$ is totally symmetric, we see that
$$
\sum_kC_{kkk}^2+3\sum_{i\ne k}C_{iik}^2=|C|^2-6C_{123}^2.
$$
Summing over $k$, we then obtain
$$
|\tr C|^2=\sum_k\big(\sum_iC_{iik}\big)^2\leq\tfrac53\left(|C|^2-6C_{123}^2\right)\leq\tfrac53|C|^2.
$$
On the other hand, we have
$$
\begin{aligned}
 |C|^2&=\sum_{i,j,k}(P_{ijk}+P_{jki}+P_{kij})^2=3\Big(|P|^2+2\sum_{i,j,k}P_{ijk}P_{jik}\Big),
\end{aligned}
$$
where the symmetry $P_{ijk}=P_{ikj}$ and a relabelling of the indices
have been used in the last equality. Consequently, by using $\tr C=\nabla s$ we get
\begin{equation}\label{eq:PS}
\begin{aligned}
|P|^2+2\sum_{i,j,k}P_{ijk}P_{jik}\geq\tfrac15|\nabla s|^2.
\end{aligned}
\end{equation}

Using \eqref{eqn:2.3}, we obtain
\begin{equation}\label{eqn:3.9}
\begin{aligned}
|W|^2&=\sum_{j,a,b,c,d}(\sum_i\omega_{iab}\omega_{icd})P_{abj}P_{cdj}\\
&=\sum_{a,b,c,d,j}(\delta_{ac}\delta_{bd}-\delta_{ad}\delta_{bc})P_{abj}P_{cdj}\\
&=|P|^2-\sum_{i,j,k}P_{ijk}P_{jik}.
\end{aligned}
\end{equation}

Similarly, by using \eqref{eqn:2.3}, \eqref{eqn:3.2} and \eqref{eq:P-trace}, we have
$$
\begin{aligned}
\sum_{i,j}\omega_{kij}W_{ij}
&=\sum_{j,a,b}(\delta_{ja}\delta_{kb}-\delta_{jb}\delta_{ka})P_{abj}\\
&=\sum_jP_{jkj}-\sum_jP_{kjj}\\
&=\sum_jP_{jjk}=\tfrac12s_{,k},\ \ \forall\,k.
\end{aligned}
$$

Moreover,
$$
\begin{aligned}
\sum_k\Big(\sum_{i,j}\omega_{kij}W_{ij}\Big)^2
&=\sum_{i,j,a,b}(\delta_{ia}\delta_{jb}-\delta_{ib}\delta_{ja})W_{ij}W_{ab}\\
&=|W|^2-\sum_{i,j}W_{ij}W_{ji}.
\end{aligned}
$$
Thus, by the definition of $A$ and \eqref{eqn:3.9},
$$
\begin{aligned}
|A|^2 &=\tfrac14\sum_{i,j}(W_{ij}+W_{ji})^2\\
&=|W|^2-\tfrac12\sum_k\Big(\sum_{i,j}\omega_{kij}W_{ij}\Big)^2\\
&=|P|^2-\sum_{i,j,k}P_{ijk}P_{jik}-\tfrac18|\nabla s|^2.
\end{aligned}
$$
This combining with \eqref{eq:PS} gives
$$
\begin{aligned}
 |P|^2+3\sum_{i,j,k}P_{ijk}P_{jik}
 ={}&\tfrac43\Big(
 |P|^2+2\sum_{i,j,k}P_{ijk}P_{jik}\Big)\\
 &-\tfrac13\Big(
 |P|^2-\sum_{i,j,k}P_{ijk}P_{jik}\Big)\\
 \geq{}&\tfrac9{40}|\nabla s|^2-\tfrac13|A|^2.
\end{aligned}
$$
We have completed the proof of Lemma \ref{lem:3.2}.
\end{proof}

It remains to estimate the terms $|A|$ and $\langle A,S\rangle$.
These estimates are given in the
following lemma.
\begin{lemma}\label{lem:3.3}
With the notation above, we have
\begin{align}
|A|^2&\leq\tfrac5{12}s|T|^2,\label{eqn:3.10}\\[2mm]
\langle A,S\rangle^2
&\leq\tfrac1{12}(-4s^3+17sq-15r)|T|^2.\label{eqn:3.11}
\end{align}
\end{lemma}

\begin{proof}
First, let $B$ be symmetric and trace-free. Choose an oriented orthonormal
frame with $B=\operatorname{diag}(b_1,b_2,b_3)$, and put
$$
d_1=b_3-b_2,\qquad d_2=b_1-b_3,\qquad d_3=b_2-b_1.
$$

Let all cyclic sums below run over $(i,j,k)=(1,2,3),(2,3,1),(3,1,2)$.
The symmetry of $T$ gives
$$
\begin{aligned}
\langle A,B\rangle\,
&=\sum_ib_iA_{ii}=\sum_ib_iW_{ii}=\sum_{i,a,b,m,n}b_i\omega_{iab}h_{amn}T_{mnbi}\\
&=\mathop\mathfrak{S}_{ijk}\sum_{m,n}b_i\big(h_{jmn}T_{mnki}-h_{kmn}T_{mnji}\big)\\
&=\mathop\mathfrak{S}_{ijk}b_i(X_j-X_k)=\mathop\mathfrak{S}_{ijk}d_iX_i,
\end{aligned}
$$
where we set
$$
\begin{aligned}
X_i:&=\sum_{m,n}h_{imn}T_{mnjk}\\
&=h_{iii}T_{iijk}+h_{ijj}T_{jjjk}+h_{ikk}T_{jkkk}\\
&\ \ \ +2h_{iij}T_{ijjk}+2h_{iik}T_{ijkk}+2h_{ijk}T_{jjkk}.
\end{aligned}
$$
It follows that to give an expression by components of $T$ we obtain
\begin{equation}\label{eqn:3.12}
\begin{aligned}
\langle A,B\rangle &=\mathop\mathfrak{S}_{ijk}d_i\Big\{h_{iii}T_{iijk}+h_{ijj}T_{jjjk}+h_{ikk}T_{jkkk}\\
&\hspace{20mm}+2h_{iij}T_{ijjk}+2h_{iik}T_{ijkk}+2h_{ijk}T_{jjkk}\Big\}\\
&=\mathop\mathfrak{S}_{ijk}\Big\{ d_i h_{ijj}T_{jjjk}+d_i h_{ikk}T_{jkkk}+2d_i h_{123}T_{jjkk}\\
 &\hspace{20mm}+(d_i h_{iii}+2d_j h_{ijj}+2d_k h_{ikk})T_{iijk}\Big\}.
\end{aligned}
\end{equation}
By the symmetry of $T$, we have
$$
|T|^2=\sum_{i,j,k,l}T_{ijkl}^2=\sum_iT_{iiii}^2+\mathop\mathfrak{S}_{ijk}\bigl(4T_{jjjk}^2+4T_{jkkk}^2
          +6T_{jjkk}^2+12T_{iijk}^2\bigr).
$$
Thus, from \eqref{eqn:3.12}, using the weighted Cauchy-Schwarz inequality
$$
\Big(\sum_\alpha c_\alpha t_\alpha\Big)^2\le\Big(\sum_\alpha\tfrac{c_\alpha^2}{m_\alpha}\Big)\Big(\sum_\alpha m_\alpha t_\alpha^2\Big),\ \ \text{for}\ m_\alpha>0,
$$
and setting
$$
\begin{aligned}
Q(B):=\mathop\mathfrak{S}_{ijk}\Big\{
 &\tfrac14d_i^2(h_{ijj}^2+h_{ikk}^2)
 +\tfrac23d_i^2h_{123}^2\\
 &+\tfrac1{12}
 (d_i h_{iii}+2d_j h_{ijj}+2d_k h_{ikk})^2\Big\},
\end{aligned}
$$
we obtain
\begin{equation}\label{eqn:3.13}
\begin{aligned}
\langle A,B\rangle^2
&\leq Q(B)\mathop\mathfrak{S}_{ijk}\bigl\{4T_{jjjk}^2+4T_{jkkk}^2
          +6T_{jjkk}^2+12T_{iijk}^2\bigr\}\\
&=Q(B)\Big(|T|^2-\sum_iT_{iiii}^2\Big)
\leq Q(B)|T|^2.
\end{aligned}
\end{equation}

To simplify $Q(B)$ we first make the following claim.

\vskip2mm\noindent
{\bf Claim 3.1.} Put $v_i=\sum_{a,b}B_{ab}h_{abi}$. Then
\begin{equation}\label{eqn:3.14}
Q(B)=\tfrac5{12}s|B|^2-\tfrac14\operatorname{tr}(B^2S)-\tfrac12|v|^2.
\end{equation}
\vskip2mm\noindent
{\it Proof of Claim 3.1}.
Put $B_{ij}=b_i\delta_{ij}$, so that $v_i=\sum_a b_a h_{aai}$. For each cyclic triple $(i,j,k)$,
write
$$
x_i=h_{iii},\qquad y_i=h_{ijj}-h_{ikk}.
$$
Trace-freeness gives
$$
h_{ijj}=\tfrac12(-x_i+y_i),\qquad h_{ikk}=\tfrac12(-x_i-y_i).
$$
By using $b_i+b_j+b_k=0$ and $d_i=b_k-b_j$, we can get
$$
3b_i^2+d_i^2=2|B|^2,\qquad \sum_i d_i^2=3|B|^2.
$$
Moreover, it holds that
$$
v_i=\tfrac12(3b_i x_i-d_i y_i),
$$
$$
d_i h_{iii}+2d_j h_{ijj}+2d_k h_{ikk}=2d_i x_i+3b_i y_i.
$$
By definitions, we have the following relation:
$$
\begin{aligned}
s&=\tfrac12\sum_i(5x_i^2+3y_i^2)+6h_{123}^2,\\
\operatorname{tr}(B^2S)&=\tfrac12\sum_i\bigl[
(|B|^2+2b_i^2)x_i^2+|B|^2y_i^2-2b_i d_i x_i y_i\bigr]+2|B|^2h_{123}^2,\\
6|v|^2&=\tfrac32\sum_i(3b_i x_i-d_i y_i)^2,\\
12Q(B)&=\sum_i\left[\tfrac32d_i^2(x_i^2+y_i^2)+(2d_i x_i+3b_i y_i)^2\right]+24|B|^2h_{123}^2.
\end{aligned}
$$
Calculating $12Q(B)+3\operatorname{tr}(B^2S)+6|v|^2$, we easily see that
the coefficient of each $x_i y_i$ vanishes:
$$
12b_i d_i-3b_i d_i-9b_i d_i=0;
$$
whereas the coefficients of $x_i^2$ and $y_i^2$ can be simplified, respectively, to be:
$$
\tfrac{11}{2}d_i^2+\tfrac{33}{2}b_i^2+\tfrac32|B|^2=\tfrac{25}{2}|B|^2,
$$
$$
3d_i^2+9b_i^2+\tfrac32|B|^2=\tfrac{15}{2}|B|^2.
$$
Hence we obtain
$$
\begin{aligned}
12Q(B)+3\operatorname{tr}(B^2S)+6|v|^2
&=\tfrac52|B|^2\sum_i(5x_i^2+3y_i^2)
  +30|B|^2h_{123}^2\\
&=5s|B|^2.
\end{aligned}
$$
This verifies the assertion of Claim 3.1.\qed

\begin{remark}\label{rm:3.1}
{\rm $Q(B)$ in \eqref{eqn:3.14} is invariant in an arbitrary orthonormal
frame and the preceding Cauchy--Schwarz estimate \eqref{eqn:3.13} holds for every
symmetric trace-free $B$.}
\end{remark}

Following \eqref{eqn:3.14} and the fact
$$
\operatorname{tr}(B^2S)=\sum_{i,a,b}\Big(\sum_j B_{ij}h_{abj}\Big)^2\geq0,
$$
we can see that
\begin{equation}\label{eqn:3.15}
0\leq Q(B)\leq\tfrac5{12}s|B|^2.
\end{equation}
Then, taking $B=A$ we derive that
$$
|A|^4=\langle A,A\rangle^2\leq Q(A)|T|^2\leq\tfrac5{12}s|A|^2|T|^2,
$$
This immediately implies the inequality \eqref{eqn:3.10}.

\vskip2mm
To prove \eqref{eqn:3.11}, as for $h=0$ the assertion is
immediate, we shall assume $h\not=0$ and choose $B=S-\tfrac13sg$.
Since $h$ is trace-free, we have
$$
v_i=\sum_{a,b}B_{ab}h_{abi}=\sum_{a,b}S_{ab}h_{abi}.
$$
$$
|B|^2=q-\tfrac13s^2,\quad \operatorname{tr}(B^2S)=r-\tfrac23sq+\tfrac19s^3.
$$
The formula \eqref{eqn:3.14} for $Q(B)$ therefore becomes
\begin{equation}\label{eqn:3.16}
12Q(B)=7sq-2s^3-3r-6|v|^2.
\end{equation}

Now, we shall compute the above expression directly in a frame adapted to $h$.

We choose $e_1$ at a minimum of $x\mapsto h(x,x,x)$ on the unit sphere. Since the
cubic is odd and nonzero, $h_{111}=-a$ with $a>0$. For $\alpha=2,3$, the directional
derivative along $\cos z\,e_1+\sin z\,e_\alpha$ is $3h_{11\alpha}$ at $z=0$.
Hence $h_{112}=h_{113}=0$. Diagonalize $h(e_1,\cdot,\cdot)$ on $e_1^\perp$, ordering
the eigenvalues so that $h_{122}\geq h_{133}$ and retaining the orientation.
Thus $h_{123}=0$, and trace-freeness gives
\begin{equation}\label{eqn:3.17}
\begin{aligned}
h_{111}&=-a,& h_{122}&=\tfrac12(a+b),&
h_{133}&=\tfrac12(a-b),\\
h_{222}&=u,& h_{223}&=w,& h_{233}&=-u,& h_{333}&=-w,
\end{aligned}
\end{equation}
where $a>0$ and $b\geq0$. The other components follow by symmetry or vanish.
Put $t=u^2+w^2$. Then by $S_{ij}=\sum_{a,b}h_{abi}h_{abj}$ we have
\begin{equation}\label{eqn:3.18}
S=\begin{pmatrix}
(3a^2+b^2)/2&bu&bw\\
bu&(a+b)^2/2+2t&0\\
bw&0&(a-b)^2/2+2t
\end{pmatrix},
\end{equation}
and
$$
v_1=-a(a^2-b^2-2t),\quad v_2=b(3a+b)u,\quad v_3=b(3a-b)w.
$$
It follows that
$$
\begin{aligned}
|v|^2&=a^2(a^2-b^2-2t)^2+b^2(3a+b)^2u^2+b^2(3a-b)^2w^2\\
&=a^2(a^2-b^2)^2+(-4a^4+13a^2b^2+b^4)t+4a^2t^2+6ab^3(u^2-w^2).
\end{aligned}
$$
By \eqref{eqn:3.18} and the definition $s=\tr S$, $q=\tr(S^2)$,
$r=\tr(S^3)$ we can derive that
\begin{equation}\label{eqn:3.19}
\begin{aligned}
s&=\tfrac12(5a^2+3b^2)+4t,\\[1mm]
q&=\tfrac14(11a^4+18a^2b^2+3b^4)+(4a^2+6b^2)t+8t^2,\\[1mm]
r&=\tfrac18(29a^6+57a^4b^2+39a^2b^4+3b^6)\\[1mm]
 &\quad+(3a^4+24a^2b^2+6b^4)t+(12a^2+18b^2)t^2+16t^3+3ab^3(u^2-w^2).
\end{aligned}
\end{equation}
Substituting these expressions into \eqref{eqn:3.16} yields
$$
\begin{aligned}
12Q(B)
&=12t(a^2-2t)^2\\
&\quad+3b^2\bigl[
14a^4+2a^2b^2+(2b^2-19a^2)t
+18t^2-15ab(u^2-w^2)
\bigr]\\
&=-4s^3+17sq-15r.
\end{aligned}
$$
Therefore, we have proved
\begin{equation}\label{eqn:3.20}
0\leq Q(S-\tfrac13sg)=\tfrac1{12}(-4s^3+17sq-15r).
\end{equation}

As $\operatorname{tr}A=0$, for $B=S-\tfrac13sg$ we finally obtain
$$
\langle A,S\rangle^2=\langle A,B\rangle^2
\leq Q(B)|T|^2=\tfrac1{12}(-4s^3+17sq-15r)|T|^2,
$$
which proves \eqref{eqn:3.11}.
\end{proof}

For $s>0$, we put
\begin{equation}\label{eqn:3.23}
\begin{aligned}
F(s,q,r)=-\tfrac{11}{36}s^3+\tfrac{55}{36}sq-3r+\tfrac5{24}s^2+\tfrac{25}{24}q+\tfrac{25r}{16s},
\end{aligned}
\end{equation}
and we set $F=0$ at $s=0$. It is easily seen that $F$ is continuous: the eigenvalues of
$S$ lie in $[0,s]$, so $0\leq q\leq s^2$ and $0\le\tfrac{r}s\leq s^2$.
Then, by combining Lemmas 3.1\,$\sim$\,3.3 with \eqref{eqn:2.11} and
\eqref{eqn:2.14}, we can obtain the following key proposition.

\begin{proposition}\label{prop:3.1}
Every closed Lagrangian submanifold $M^3$ in $\mathbb S^6(1)$ satisfies
\begin{equation}\label{eqn:3.24}
0\geq\int_M\Big(\tfrac{13}{180}|\nabla s|^2+\tfrac1{180}s|T|^2+F(s,q,r)\Big)\dd M.
\end{equation}
If $M$ is connected and $F\geq0$, then $T=0$.
\end{proposition}

\begin{proof}
Adding $\tfrac7{36}$ times \eqref{eqn:2.11} to \eqref{eqn:2.14}, we obtain
\begin{equation}\label{eqn:3.25}
\begin{aligned}
0=\int_M\Big[&
\tfrac7{36}s|T|^2+\sum_{i,j,k}S_{ij,k}S_{ik,j}-\tfrac{11}{72}|\nabla s|^2\\
&-\tfrac{11}{36}s^3+\tfrac{55}{36}sq-3r-\tfrac{13}{48}s^2+3q
 \Big]\dd M.
\end{aligned}
\end{equation}
By \eqref{eqn:3.4}, \eqref{eqn:3.6} and \eqref{eqn:3.10},
we get
$$
\sum_{i,j,k}S_{ij,k}S_{ik,j}\geq\tfrac9{40}|\nabla s|^2-\tfrac5{36}s|T|^2
+\tfrac12\langle A,S\rangle+\tfrac1{16}(s^2-3q).
$$
On $\{s>0\}$, \eqref{eqn:3.11} gives
$$
\begin{aligned}
\tfrac12\langle A,S\rangle
&\geq-\tfrac12\left(\tfrac{-4s^3+17sq-15r}{12}\right)^{1/2}|T|\\
&\geq-\tfrac{1}{20}s|T|^2-\tfrac5{48s}(-4s^3+17sq-15r).
\end{aligned}
$$
Thus we have
\begin{equation}\label{eqn:3.26}
\begin{aligned}
\tfrac7{36}s|T|^2+\sum_{i,j,k}S_{ij,k}S_{ik,j}
\geq{}&\tfrac1{180}s|T|^2+\tfrac9{40}|\nabla s|^2\\
&+\tfrac1{48}\left(23s^2-94q+75\tfrac{r}{s}\right).
\end{aligned}
\end{equation}

At a zero of $h$, both $\nabla S$ and $\nabla s$ vanish.
Together with $0\leq r/s\leq s^2$, this shows that the inequality
extends continuously across $\{s=0\}$.

Substituting \eqref{eqn:3.26} into \eqref{eqn:3.25} we immediately
prove \eqref{eqn:3.24}.

If $F\geq0$, the integrand in \eqref{eqn:3.24} is nonnegative
and continuous, so $\nabla s=0$ and $s|T|^2=0$ everywhere.
Since $M$ is connected, $s$ is constant.
If $s>0$, then $T=0$; if $s=0$, then $h=0$ and again $T=0$.
\end{proof}

\section{Proofs of the pinching theorems}\label{sect:4}

To prove the two pinching theorems, we shall verify that $F\geq0$ under the pinching 
condition of $|h|^2$ and $\Ric$, respectively. Then Proposition~\ref{prop:3.1} gives $T=0$, and 
Lemma~\ref{lem:2.2} determines the submanifold.

To prove Theorem \ref{thm:1.2}, we first prove two useful inequalities
involving $s,q,r$ as follows.

\begin{lemma}\label{lem:3.4}
Every symmetric trace-free cubic on a Euclidean three-space satisfies
\begin{align}
s^3+3sq-10r&\geq0,\label{eqn:3.21}\\
 -4s^3+17sq-15r&\geq0.\label{eqn:3.22}
\end{align}
\end{lemma}

\begin{proof}
The second inequality is given by \eqref{eqn:3.20}.

For the first inequality, the case $h=0$ is immediate. Otherwise, use
the frame before such that \eqref{eqn:3.17} holds, where $a>0$, $b\geq0$,
and $t=u^2+w^2$. Substituting \eqref{eqn:3.19} gives
$$
\begin{aligned}
 \tfrac13(s^3+3sq-10r)
 ={}&b^2(a^2-b^2)^2+(36a^4-11a^2b^2+b^4)t\\
 &+36a^2t^2-10ab^3(u^2-w^2).
\end{aligned}
$$
Since $u^2-w^2=t-2w^2$, the above expression can be rewritten to be
$$
\begin{aligned}
 \tfrac13(s^3+3sq-10r)={}&[6at-b(b^2-a^2)]^2+(2a-b)^2(3a+b)^2t+20ab^3w^2,
\end{aligned}
$$
which is obvious nonnegative.
\end{proof}

\noindent
{\bf Completion of Theorem~\ref{thm:1.2}'s proof}.

For $s>0$, the function in \eqref{eqn:3.23} has the factorization
$$
\begin{aligned}
F(s,q,r)={}&\tfrac{25-8s}{240s}(2s^3+10sq+15r)\\
&+\tfrac7{60}(s^3+3sq-10r)+\tfrac4{45}(-4s^3+17sq-15r).
\end{aligned}
$$

If $0<s\leq25/8$, the first term in this decomposition is nonnegative,
since $q,r\geq0$. The other two are nonnegative by \eqref{eqn:3.21} and
\eqref{eqn:3.22}. With the continuous definition $F=0$ at $s=0$, 
Proposition~\ref{prop:3.1} gives $T=0$. Lemma~\ref{lem:2.2} gives 
$K\ge\tfrac1{16}$, so the global classification of \cite[Main Theorem]{DVV} 
applies. The constant sectional curvature $\tfrac1{16}$ sphere is excluded 
because it satisfies $s=\tfrac{45}8>\tfrac{25}8$. Thus $M$ is totally geodesic 
or congruent to the Berger sphere of Dillen-Verstraelen-Vrancken, on which 
$s=\tfrac{25}8$.
\qed

\vskip2mm
To prove Theorem \ref{thm:1.1}, we first prove the following lemma.

\begin{lemma}\label{lem:4.1}
Diagonalizing $S$ at the point, say $S_{ij}=\mu_i\delta_{ij}$,
and write $s=\sum_i\mu_i$, $q=\sum_i\mu_i^2$, $r=\sum_i\mu_i^3$.
If $0\leq\mu_i\leq\tfrac{15}8$ and $s^3+3sq-10r\geq0$, then $F(s,q,r)\geq0$.
\end{lemma}
\begin{proof}
The case $s=0$ is immediate. Suppose that $s>0$. Then by symmetry we may 
assume w.l.g that $\mu_1\geq\mu_2\geq\mu_3\geq0$. For $\lambda>0$, we have
$$
\begin{aligned}
\tfrac{F(\lambda s,\lambda^2q,\lambda^3r)}{\lambda^2}
={}&-\tfrac{\lambda}{36}(11s^3-55sq+108r)\\
&+\tfrac5{24}s^2+\tfrac{25}{24}q+\tfrac{25r}{16s}.
\end{aligned}
$$
Differentiating gives
$$
\begin{aligned}
\tfrac{\mathrm d}{\mathrm d\lambda}\left(\tfrac{F(\lambda s,\lambda^2q,\lambda^3r)}{\lambda^2}\right)
={}&-\tfrac{1}{36}(11s^3-55sq+108r)<0.
\end{aligned}
$$
Here, $11s^3-55sq+108r$ is strictly positive since $s>0$ and
$$
\begin{aligned}
11s^3-55sq+108r
&=\sum_i\mu_i(108\mu_i^2-55s\mu_i+11s^2),\\
108\mu_i^2-55s\mu_i+11s^2
&=108\left(\mu_i-\tfrac{55s}{216}\right)^2
  +\tfrac{1727}{432}s^2>0.
\end{aligned}
$$
Thus, we have
$$
F(s,q,r)\geq\lambda^{-2}F(\lambda s,\lambda^2q,\lambda^3r),\ \ \text{for}\ \lambda\ge1.
$$
It therefore suffices to prove the assertion $F(s,q,r)\ge0$ for $\mu_1=\tfrac{15}8$, 
if necessary by taking $\lambda=\tfrac{15}{8\mu_1}\geq1$. Let us write
$$
t=\tfrac{\mu_2+\mu_3}{\mu_1},\qquad
v=\tfrac{\mu_2\mu_3}{\mu_1^2}.
$$
Then
$$
0\leq t\leq2,\qquad 0\leq v\leq\tfrac{t^2}{4},\qquad v\geq t-1,
$$
where the last inequality follows from $(1-\mu_2/\mu_1)(1-\mu_3/\mu_1)\geq0$.
We have
\begin{equation}\label{eqn:4.3}
s=\tfrac{15}8(1+t),\quad
q=\left(\tfrac{15}8\right)^2(1+t^2-2v),\quad
r=\left(\tfrac{15}8\right)^3(1+t^3-3tv).
\end{equation}
Thus $s^3+3sq-10r\geq0$ is equivalent to
\begin{equation}\label{eqn:4.4}
(4t-1)v\geq(1+t)(t-1)^2.
\end{equation}
If $t\leq\tfrac14$, in \eqref{eqn:4.4} the left-hand side is nonpositive and 
the right-hand side is positive, a contradiction. Thus we have $t>\tfrac14$ 
and $v\leq t^2/4$. From \eqref{eqn:4.4} we obtain
$$
0\leq(4t-1)\tfrac{t^2}{4}-(1+t)(t-1)^2=\tfrac{(3t-2)(t+2)}4,
$$
so we get $t\geq\tfrac23$.

Substituting \eqref{eqn:4.3} into $F(s,q,r)$ yields
$$
F(s,q,r)=\tfrac{375}{1024(1+t)}\mathcal P(t,v),
$$
where
$$
\mathcal P(t,v)
=-32t^4+6t^3+38t^2-5t-5+(107t^2-13t-75)v.
$$
For fixed $t$, this polynomial is affine in $v$, so to show 
$\mathcal P(t,v)\ge0$ it suffices to check the two endpoints 
of an interval containing all admissible values of $v$. We consider
two cases:

$\bullet$\ If $\tfrac23\leq t\leq1$, then
$$
\tfrac{(1+t)(t-1)^2}{4t-1}\leq v\leq\tfrac{t^2}{4}.
$$
At the endpoints we have
$$
\begin{aligned}
\mathcal P\left(t,\tfrac{t^2}{4}\right)
&=\tfrac{(2-t)(3t-2)(7t^2+15t+5)}4,\\
\mathcal P\left(t,\tfrac{(1+t)(t-1)^2}{4t-1}\right)
&=\tfrac{(1+t)(3t-2)(35-6t-19t^2-7t^3)}{4t-1}.
\end{aligned}
$$
Obviously, both of the above two expressions are nonnegative for $\tfrac23\leq t\leq1$.

$\bullet$\ If $1\leq t\leq2$, then $t-1\leq v\leq t^2/4$.
The upper endpoint value is nonnegative by the factorization above,
and the lower endpoint value is
$$
\mathcal P(t,t-1)=(2-t)(32t^3-49t^2-16t+35)\geq0,
$$
since
$$
32t^3-49t^2-16t+35
=32(t-1)^3+47\left(t-\tfrac{56}{47}\right)^2+\tfrac{13}{47}>0.
$$

In conclusion, we have completed the proof of Lemma \ref{lem:4.1}.
\end{proof}

\noindent
{\bf Completion of Theorem~\ref{thm:1.1}'s proof}.

By \eqref{eq:ricci}, the assumption $\Ric\geq\tfrac18$ and the positive 
semidefiniteness of $S$ imply that all eigenvalues of $S$ lie in $[0,\tfrac{15}8]$.
Moreover, \eqref{eqn:3.21} gives $s^3+3sq-10r\geq0$. Lemma~\ref{lem:4.1} 
gives $F\geq0$, so Proposition~\ref{prop:3.1} yields $T=0$. By Lemma 
\ref{lem:2.2}, all sectional curvatures satisfy $K\ge\tfrac1{16}$. Since 
$M$ is closed, it is complete, and the global classification of
\cite[Main Theorem]{DVV} applies. Thus $M$ is either totally geodesic, or 
has constant sectional curvature $\tfrac1{16}$, or is congruent
to the Berger sphere of Dillen-Verstraelen-Vrancken.
\qed

\vskip 2mm
\noindent{\bf Acknowledgments}.
Z. Hu was supported by the National NSF of China (Grant No. 12171437).
L. Sun acknowledges support from the National Natural Science Foundation of China
(Grant Nos.~12571055 and 12671062), the Natural Science Foundation of Hunan Province
(Grant No.~2026JJ20014), the 111 Project (Grant No.~D23017), and the Program for
Science and Technology Innovative Research Team in Higher Educational Institutions
of Hunan Province, China.
J. Yin was supported by the National NSF of China (Grant No. 12201138) and NSF of Henan Province (Grant No. 262300421869).

\vskip2mm
\noindent{\bf AI assistance statement.} The authors used AI models to assist 
with proving Lemma 3.3; the main ideas and all final checks remain the sole
responsibility of the human authors.

\vskip2mm
\noindent{\bf Competing Interests}
The authors declare that there are no conflicts of interest.


\end{document}